\documentclass[11 pt]{amsart}
\usepackage{amscd,amsfonts,amssymb,amsmath}
\usepackage{hyperref}
\usepackage{epsfig}
\usepackage{mathtools}
\usepackage{tabu}
\usepackage{tikz-cd}
\newtheorem{theorem}{Theorem}[section]
\newtheorem{corollary}[theorem]{Corollary}
\newtheorem{lemma}[theorem]{Lemma}
\newtheorem{proposition}[theorem]{Proposition}
\theoremstyle{definition}

\numberwithin{equation}{subsection}

\usepackage[all,cmtip]{xy}

\usepackage{graphicx} 
\newcommand{\IA}{\operatorname{IA}}

\newcommand{\Aut}{\operatorname{Aut}}

\newcommand{\Ker}{\operatorname{Ker}}

\newcommand{\Inn}{\operatorname{Inn}}
\newcommand{\Out}{\operatorname{Out}}

\newcommand{\C}{\operatorname{C}}

\newcommand{\Z}{\operatorname{Z}}

\begin{document}

\title{Conjugacy classes and automorphisms of twin groups}

\author{Tushar Kanta Naik}
\address{Department of Mathematical Sciences, Indian Institute of Science Education and Research (IISER) Mohali, Sector 81,  S. A. S. Nagar, P. O. Manauli, Punjab 140306, India.}
\email{mathematics67@gmail.com, tushar@iisermohali.ac.in}
\author{Neha Nanda}
\email{nehananda94@gmail.com, nehananda@iisermohali.ac.in}
\author{Mahender Singh}
\email{mahender@iisermohali.ac.in}


\begin{abstract}
The twin group $T_n$ is a right angled Coxeter group generated by $n-1$ involutions and the pure twin group $PT_n$ is the kernel of the natural surjection from $T_n$ onto the symmetric group on $n$ symbols. In this paper, we investigate some structural aspects of these groups. We derive a formula for the number of conjugacy classes of involutions in $T_n$, which quite interestingly, is related to the well-known Fibonacci sequence.  We also derive a recursive formula for the number of $z$-classes of involutions in $T_n$. We give a new proof of the structure of $\Aut(T_n)$ for $n \ge 3$, and show that $T_n$ is isomorphic to a subgroup of $\Aut(PT_n)$ for $n \geq 4$. Finally, we construct a representation of $T_n$ to $\Aut(F_n)$ for $n \ge 2$.
\end{abstract}
\maketitle

\section{Introduction}\label{introduction}
The twin group $T_n$, $n \ge 2$, is generated by $n-1$ involutions such that two generators commute if and only if they are not adjacent, and the pure twin group $PT_n$ is the kernel of the natural surjection from $T_n$ onto the symmetric group $S_n$ on $n$ symbols.  Twin groups form a special class of right angled Coxeter groups and appeared in the work of Shabat and Voevodsky \cite{sv}, who referred them as Grothendieck cartographical groups. Later, these groups appeared in the work of Khovanov \cite{Khovanov} under the name twin groups, who gave a geometric interpretation of these groups similar to the one for classical braid groups. Consider configurations of $n$ arcs in the infinite strip $\mathbb{R} \times  [0,1]$ connecting $n$ marked points on each of the parallel lines $\mathbb{R} \times \{1\}$ and $\mathbb{R} \times \{0\}$ such that each arc is monotonic and no three arcs have a point in common. Two such configurations are equivalent if one can be deformed into the other by a homotopy of such configurations in $\mathbb{R} \times [0,1]$ keeping the end points of arcs fixed. An equivalence class under this equivalence is called a \textit{twin}. The product of two twins can be defined by placing one twin on top of the other, similar to the product in the braid group $B_n$. The collection of all twins with $n$ arcs under this operation forms a group isomorphic to $T_n$. Taking the one point compactification of the plane, one can define the closure of a twin on a $2$-sphere analogous to the closure of a braid in $\mathbb{R}^3$. A \textit{doodle} on a closed oriented surface is a finite collection of piecewise linear closed curves without triple intersections. It is not difficult to show that a closure of a twin on a $2$-sphere is a doodle. Doodles on a 2-sphere were first introduced by Fenn and Taylor \cite{Fenn-Taylor}, and the notion was extended to immersed circles in a 2-sphere by Khovanov \cite{Khovanov}. He proved an analogue of the classical Alexander Theorem for doodles, that is, every oriented doodle on a $2$-sphere is closure of a twin. Recently, Gotin \cite{Gotin} proved an analogue of the Markov Theorem for doodles and twins.  Bartholomew-Fenn-Kamada-Kamada \cite{fenn} extended the study of doodles  to immersed circles in a closed oriented surface of any genus, which can be thought of as virtual links analogue for doodles. Recently, in \cite{bfkk}, they constructed an invariant of virtual doodles by coloring their diagrams using some special type of algebra.
\par

Our aim in this paper is to investigate twin and pure twin groups from an algebraic point of view, a direction which has recently attracted a lot of attention. In a recent paper \cite{bardakov}, Bardakov-Singh-Vesnin proved that $PT_n$ is free for $n = 3,4$ and not free for $n \geq 6$. It was conjectured that  $PT_5$ is also a free group of rank $31$, and the same has been established recently by Gonz\'alez-Le\'on-Medina-Roque \cite{Jesus}. A lower bound for the number of generators of $PT_n$ is given in \cite{Anyon} and an upper bound is given in \cite{bardakov}. It is worth noting that authors in \cite{Anyon} refer twin and pure twin groups as  traid and pure traid groups, respectively. Genevois informed us that pure twin groups belong to the class of so called \textit{diagram groups} \cite{Genevois, Guba}. Description of $PT_6$ has been obtained recently by Mostovoy and Roque-M\'arquez \cite{Mostovoy}. It has been proven that $PT_6$ is a free product of the free group $F_{71}$ and 20 copies of the free abelian group $\mathbb{Z} \oplus \mathbb{Z}$. A complete presentation of $PT_n$ for $n \ge 7$ is still not known and seems challenging to describe.
\par

We explore conjugacy classes of involutions, centralisers, automorphisms and representations of twin groups. Kaul-White \cite{Kaul} determined centralisers of some special type of involutions by studying the maximal complete subgraphs of graphs associated to right angled Coxeter groups. We give a precise formula for the number of conjugacy classes of involutions in $T_n$. M{\"u}hlherr \cite{Mullher} and Tits \cite{Tits} studied automorphisms of Coxeter groups of graph-universal type by studying associated graphs of Coxeter systems. Twin groups are special type of graph-universal Coxeter groups, and the structure of their automorphism groups was first obtained by James \cite{James}.
\par

The paper is organised as follows. In Section \ref{basic}, we recall definition of twin and pure twin groups and set basic ideas from combinatorial group theory needed in the rest of the paper. In Section \ref{conjugacy-problem}, we investigate the conjugacy problem in twin groups. In section \ref{Involution},  we derive a formula for the number of conjugacy classes of involutions in $T_n$, which quite interestingly, is related to the well-known Fibonacci sequence (Theorem \ref{theoremrho}). In Section \ref{z-classes}, we investigate $z$-classes (conjugacy classes of centralisers of elements) in twin groups and derive a recursive formula for the number of $z$-classes of involutions (Theorem \ref{thm-zclass}).  In Section \ref{Automorphism}, we determine $\Aut(T_n)$ for all $n \ge3$ (Theorem \ref{thm4.1}). Although this result is known from \cite{James},  our approach is  elementary and we also give some applications. More precisely, we deduce that $PT_n$ is not characteristic in $T_n$ for $n \ge 4$ and that $T_n$ is isomorphic to a subgroup of $\Aut(PT_n)$ for $n \geq4$, which answers a question from \cite{bardakov}. Further, we also prove that the group of $\IA$ and normal automorphisms of $T_n$ is precisely the group of inner automorphisms of $T_n$. This is an analogue of a similar result for braid groups due to Neshchadim \cite{Neshchadim}.  Finally, in Section \ref{representations}, we construct a representation of $T_n$ to $\Aut(F_n)$ (Theorem \ref{faithful-representation}).
\par

We conclude the introduction by setting some notations. For elements $g, h$ of a group $G$, we use the notation  $[ g, h ]:= g^{-1}h^{-1}gh$, $g^G := $ the conjugacy class of $g$ in $G$, $\C_G(g):=$ the centraliser of $g$ in $G$ and $\widehat{g} := $ the inner automorphism of $G$ induced by $g$, that is, $\widehat{g}(x) = g^{-1}xg$ for all $x \in G$.
\bigskip

\section{Preliminaries}\label{basic}
For an integer $n \ge 2$, the \textit{twin group} $T_n$ is defined as the group generated by the set $$S = \{s_1, s_2, \dots, s_{n-1}\}$$ and satisfying the following defining relations
$$s_i^{2} = 1~\text{for all}~i,~ \textrm{and } s_is_j = s_js_i \text{ whenever } |i-j| \geq 2.$$
It follows that $T_2 \cong \mathbb{Z}_2$ and $T_3 \cong \mathbb{Z}_2 *\mathbb{Z}_2$, the infinite dihedral group. Some authors also set $T_1$ as the trivial group. Let $S_n$ be the symmetric group on $n$ symbols. Then there is a natural homomorphism $$\pi: T_n\rightarrow S_n,$$
which maps each generator $s_i$ to the transposition $(i,\;i+1)$. The kernel of this homomorphism is called the {\it pure twin group} and is denoted by $PT_n$. Note that $PT_2=1$ and $PT_3 =\big\langle (s_1s_2)^3 \big\rangle \cong \mathbb{Z}$. In \cite[Theorem 2 and Theorem 3]{bardakov}, it has been shown that $PT_4\cong F_7$ and $PT_n$ is not free for $n\geq 6$. In the same paper, it is conjectured that $PT_5\cong F_{31}$, which has been recently established in \cite{Jesus}. Further, in a recent paper \cite{Mostovoy}, it has been shown that $PT_6 \cong F_{71} * \big(*_{20} (\mathbb{Z} \oplus \mathbb{Z})\big)$.\\

It is evident from the presentation of $T_n$ that an element of $T_n$ can have more than one expression. For example, the words $s_1, s_3s_1s_3$ and $s_1s_2s_3s_5s_3s_2s_5$ represent the same element in $T_n$. In the rest of this section, we recall some ideas from combinatorial group theory that would ease our computations. Most of this section is motivated from \cite[Chapter 1]{lyndon-schupp}.

\subsection{Elementary transformations} We define three elementary transformations of a word $w\in T_n$ as follows:
\begin{itemize}
\item[(i)] \textbf{Deletion.} Replace the word $w$ by deleting a subword of the form $s_is_i$ in $w$.
\item[(ii)]\textbf{Insertion.} Replace the word $w$ by inserting a word of the form $s_is_i$ in $w$.
\item[(iii)]\textbf{Flip.} Replace a subword of $w$ of the form $s_is_j$ by $s_js_i$ whenever $\mid i-j\mid \geq 2$.
\end{itemize}

\subsection{Word equivalence and length}
We say that two words $w_1$ and $w_2$ are {\it equivalent}, written as $w_1 \sim w_2$, if there is a finite sequence of elementary transformations turning $w_1$ into $w_2$. It is easy to check that $\sim$ is an equivalence relation on $T_n$. We note that two words are equivalent if and only if both of them represent the same element of $T_n$.
\par

For a given word $w=s_{i_1} s_{i_2}\dots s_{i_k}$, let $\ell(w)=k$ be the \textit{length} of $w$.  For $1\leq i\leq n-1$, we define $\eta_i(w) :=$ number of $s_i$'s present in the expression $w$. Note that 
 $$\ell(w) =\sum_{i=1}^{n-1}\eta_i(w).$$
 
If $w_1 \sim w_2$, then $\eta_i(w_1)\equiv \eta_i(w_2)\pmod 2$ for each $1\leq i\leq n-1$, and subsequently $\ell(w_1) \equiv \ell(w_2)\pmod 2$. 
 
\subsection{Reduced words}
We say that a word $w \in T_n$ is {\it reduced} if $\ell(w)\leq \ell(w')$ for all $w'\sim w$. The existence of a reduced word in an equivalence class of a word follows from the well-ordering principle. It is possible to have more than one reduced word representing the same element. Moreover, two reduced words represent the same element if and only if one can be obtained from the other by finitely many flip transformations, for example, $s_1s_4$ and $s_4s_1$. Obviously, any two reduced words in the same equivalence class have the same length. This allows us to define  the \textit{length} of an element $w \in T_n$ as the length of a reduced word representing $w$. \\

For each $1\leq i\leq n-1$, we define the following subset of $S$;  
$$s_i^*=\big\{s_j\mid [s_i,\;s_j]\neq 1\big\}.$$ More precisely, $s_1^*=\{s_2\}$, $s_2^*=\{s_1, s_3\}$, $s_3^*=\{s_2, s_4\}, \dots ,s_{n-2}^*=\{s_{n-3}, s_{n-1}\}$ and $s_{n-1}^*=\{s_{n-2}\}.$ The following are easy observations:
\begin{itemize}
\item[(i)] $s_i\in s_j^*$ if and only if $s_j\in s_i^*$.
\item[(ii)] $[s_i,\;s_j]=1$ if and only if $s_j\notin s_i^*$.
\end{itemize}

Below is a characterisation of a reduced word in $T_n$.

\begin{lemma}\label{lem1}
A word $w$ is reduced if and only if $w$ satisfies the property that whenever two $s_{i}$'s appear in $w$ for some $1\leq i\leq n-1$, there always exists an $s_j\in s_i^*$ in between them. 
\end{lemma}

\begin{proof}
Suppose that $w$ is a reduced word and that there exist two $s_{i}$'s in $w$ such no $s_j\in s_i^*$ appears in between them. Then, by successive application of the flip transformation, we can bring the two $s_i$'s together, and then delete them by the deletion transformation. Thus, the resulting word, which is equivalent to $w$, has length strictly less than $\ell(w)$, contradicting the fact that $w$ is reduced.
\par
Conversely, suppose that the word $w$ satisfies the desired property. We note that a word obtained by flip transformations on $w$ also satisfies the desired property. Since deletion cannot be performed on words with this property, it follows that $w$ must be reduced.
\end{proof}

\subsection{Cyclic permutation}
A \textit{cyclic permutation} of a word  $w=s_{i_1} s_{i_2}\dots s_{i_k}$ (not necessarily reduced) is a word $w'$ (not necessarily distinct from $w$) of the form $s_{i_t}s_{i_{t+1}} s_{i_{t+2}}\dots s_{i_k}s_{i_1} s_{i_2}\cdots s_{i_{t-1}}$ for some $1\leq t \leq k$.  If $t=1$, then $w'=w$. It is easy to see that $w'=(s_{i_1} s_{i_2}\dots s_{i_{t-1}})^{-1}w(s_{i_1} s_{i_2}\dots s_{i_{t-1}})$ in $T_n$, that is, $w$ and $w'$ are conjugates of each other in $T_n$. 

\subsection{Cyclically reduced words}\label{cyclic permutation}
A word $w$ is called \textit{cyclically reduced} if each cyclic permutation of $w$ is reduced. It is immediate that a cyclically reduced word is reduced,  but the converse is not true. For example, $s_1s_2s_1$ is reduced but not cyclically reduced.

\begin{lemma}\label{cyc-red}
If $w$ is a cyclically reduced word and $w'$ is a word obtained from $w$ by finitely many flip transformations, then $w'$ is also cyclically reduced. 
\end{lemma}

\begin{proof}
By induction, it suffices to prove the assertion for only one flip transformation on $w$. We begin by noting that any cyclic permutation of a cyclically reduced word is again a cyclically reduced word. Thus, without loss of generality, we can assume that $w=s_{i_1} s_{i_2}s_{i_3}\dots  s_{i_k}$ and $w'=s_{i_2} s_{i_1}s_{i_3}\dots  s_{i_k}$.  We observe  that except the word $s_{i_1}s_{i_3}\dots  s_{i_k}s_{i_2}$, all other cyclic permutations of $w'$ differ by a flip transformation from some cyclic permutation of $w$, and hence are reduced. Thus, it only remains to show that the word $s_{i_1}s_{i_3}\dots  s_{i_k}s_{i_2}$ is reduced. Since $w'$ is reduced,  so are all its subwords, in particular, $s_{i_1}s_{i_3}\dots  s_{i_k}$ and $s_{i_3}\dots  s_{i_k}$ are reduced.  If $s_{i_1}s_{i_3}\dots  s_{i_k}s_{i_2}$  is not reduced, then the only reduction possible is in its subword $s_{i_3}\dots  s_{i_k}s_{i_2}$, but then the word $s_{i_3}\dots  s_{i_k}s_{i_2}s_{i_1}$ is not reduced, which is a contradiction.\end{proof}

The following result is an analogue of Lemma \ref{lem1} for cyclically reduced words.

\begin{lemma}\label{cyclically-reduced}
A reduced word $w$ is cyclically reduced if and only if we cannot obtain a word of the form $s_iw's_i$ from $w$ by applying finitely many flip transformations on $w$.
\end{lemma}

\begin{proof}
Suppose that $w$ is a cyclically reduced word and $s_iw's_i$ is obtained from $w$ by applying finitely many flip transformations. Then, by Lemma \ref{cyc-red}, $s_iw's_i$ is also cyclically reduced. Since a cyclic permutation of a cyclically reduced word is cyclically reduced, it follows that $w's_is_i$ is cyclically reduced, which is a contradiction.

Conversely, suppose that a reduced word $w$ is not cyclically reduced. That is, some cyclic permutation of $w$ is not reduced. We may assume that $w$ is of the form $w_1w_2$ so that its cyclic permutation $w_2w_1$ is not reduced. Since $w$ is reduced, both of its subwords $w_1$ and $w_2$ are also reduced. On the other hand, the word $w_2w_1$ is not reduced. This is possible only if, by applying finitely many flip transformations, $w_1$ and $w_2$ can be written in the form $s_iw_1'$ and $w_2's_i$, respectively, for some $1\leq i\leq n-1$. Thus, by applying finitely many flip transformations on the word $w=w_1w_2$, we obtain the word $s_iw_1'w_2's_i$, which is a contradiction.
\end{proof}

\begin{corollary}\label{cor1}
Each word in $T_n$ is conjugate to some cyclically reduced word.
\end{corollary}
\bigskip

\section{Conjugacy problem in twin groups}\label{conjugacy-problem}
In this section, we investigate conjugacy problem in twin groups.  In view of Corollary \ref{cor1}, it is enough to focus on cyclically reduced words to study conjugacy problem in $T_n$. The following result gives a necessary and sufficient condition for the same.

\begin{theorem}\label{conditionforconjugate}
 Suppose $w_1, w_2$ are two cyclically reduced words in $T_n$. Then $w_1$ is conjugate to $w_2$ if and only if they are cyclic permutation of each other modulo finitely many flip transformations.
\end{theorem}

\begin{proof}The converse is obvious. Let us assume that $w_1, w_2 \in T_n$ are two cyclically reduced conjugate words. Let $w_1=w^{-1}w_2w$, where $w$ is a reduced word. We need to show that $w_1$ and $w_2$ are cyclic permutation of each other modulo finitely many flip transformations. We use induction on $\ell(w)$.

Suppose $\ell(w)=1$, that is, $w=s_i$ for some $i = 1, 2, \dots, n-1$. Then $w_1=s_iw_2s_i$. Since $w_1$ is cyclically reduced, the two $s_{i}$'s should get cancelled. The following are the three possibilities:
\begin{itemize}
\item[(i)] There is cancellation in the subword $w_2s_i$.
\item[(ii)] There is cancellation in the subword $s_iw_2$.
\item[(iii)] Both the rightmost and the leftmost $s_i$ cancel each other after finitely many flip transformations.
\end{itemize}
In Case (iii), $w_1=w_2$ and we are done. In Case (i), by successive application of flip transformations on $w_2$, we obtain $w_2's_i$. This implies that  by successive application of flip transformations on the word $s_iw_2s_i$, we get $s_iw_2's_is_i$. By deletion transformation this gives $w_1=s_iw_2s_i=s_iw_2'$. Note that it is a cyclic permutation of $w_2's_i$, which we obtained by flip transformations on $w_2$. Case (ii) can be treated in the same manner.

Now suppose that $\ell(w)=k>1$, where $w=s_{i_1} s_{i_2}\dots s_{i_k}$. Then we can write $$w_1=(s_{i_1} s_{i_2}\dots s_{i_k})^{-1}w_2 (s_{i_1} s_{i_2}\dots s_{i_k}) =s_{i_k} s_{i_{k-1}}\dots s_{i_2}s_{i_1}w_2s_{i_1} s_{i_2}\dots s_{i_k}.$$ Since $w_1$ is cyclically reduced, $s_{i_k}$ should get cancelled. Following the steps of the case $k=1$, we have the following possibilities:
\begin{itemize}
\item[(i$'$)] There is cancellation of rightmost $s_{i_k}$ in the word $w_2s_{i_1} s_{i_2}\dots s_{i_k}$.
\item[(ii$'$)] There is cancellation of leftmost $s_{i_k}$ in the word $s_{i_k} s_{i_{k-1}}\dots s_{i_2}s_{i_1}w_2$.
\item[(iii$'$)]  Both the rightmost and the leftmost $s_{i_k}$ cancel each other after finitely many flip transformations. 
\end{itemize}
In Case (iii$'$) it is easy to see that $w_1=s_{i_{k-1}} s_{i_{k-2}}\cdots s_{i_2}s_{i_1}w_2s_{i_1} s_{i_2}\cdots s_{i_{k-1}}$ modulo finitely many flip transformations. Thus, we are done by induction hypothesis. For Case (ii$'$), by successive application of flip transformations on $w_2$, $s_{i_k} s_{i_{k-1}}\dots s_{i_2}s_{i_1}$ and $s_{i_1} s_{i_2}\dots s_{i_k}$, we obtain subwords $s_{i_k}w_2'$, $s_{i_{k-1}} s_{i_{k-2}}\dots s_{i_2}s_{i_1}s_{i_k}$ and $s_{i_k}s_{i_1} s_{i_2}\dots s_{i_{k-1}}$, respectively. Thus, after finitely many flip transformations, we get  $w_1=s_{i_{k-1}} s_{i_{k-2}}\dots s_{i_2}s_{i_1}s_{i_k}s_{i_k}w_2' s_{i_k} s_{i_1} s_{i_2}\dots s_{i_{k-1}}$. Consequently,  by deletion transformation, we have $$w_1=s_{i_{k-1}} s_{i_{k-2}}\dots s_{i_2}s_{i_1} w_2' s_{i_k}s_{i_1} s_{i_2}\dots s_{i_{k-1}}=(s_{i_1} s_{i_2}\dots s_{i_{k-1}})^{-1} w_2' s_{i_k} (s_{i_1} s_{i_2}\dots s_{i_{k-1}}).$$ Thus, by induction hypothesis, $w_2's_{i_k}$ (and hence $s_{i_k}w_2'$) is a cyclic permutation of $w_1$ modulo finitely many flip transformations. Since $s_{i_k}w_2'$ is obtained from $w_2$ by finitely many flip transformations, the proof of the assertion follows. Case (i$'$) can be dealt with along similar lines.
\end{proof}

\begin{corollary}
A word $w \in T_n$ is cyclically reduced if and only if $\ell(w)$ is minimal in its conjugacy class.
\end{corollary}
\bigskip

\section{Conjugacy classes of involutions in twin groups}\label{Involution}
In this section, we study conjugacy classes of involutions in $T_n$. Since conjugate elements have the same order, in view of Corollary \ref{cor1},   it suffices to study cyclically reduced involutions in $T_n$. Specifically, we derive a formula for the number of conjugacy classes of involutions in $T_n$. Quite interestingly, it is closely related to the well-known Fibonacci sequence.

\begin{proposition}\label{prop-inv}
Let $w=s_{i_1} s_{i_2}\dots s_{i_k}$ be a cyclically reduced word in $T_n$. Then $w$ is an involution if and only if $[s_{i_j},s_{i_l}]=1$ for all $1\leq j , l\leq k$.
\end{proposition}
\begin{proof}
Let us suppose that $w$ is an involution and that it does not satisfy the desired condition. Since $w$ is cyclically reduced, without loss of generality, we may assume that $w$ can be written as $w=s_iw_1s_{i+1}w_2$ such that $\eta_i(w_1)=\eta_{i+1}(w_1)=0$ for some $1\leq i\leq n-2$. Since $w$ is an involution,  we have $$w^2=s_iw_1s_{i+1}w_2s_iw_1s_{i+1}w_2=1.$$ Thus, every letter (in particular, $s_i$ and $s_{i+1}$) on the left hand side of the preceding expression should get cancelled by a finite sequence of flip and deletion transformations. But, as $w=s_iw_1s_{i+1}w_2$ is reduced, we cannot use deletion transformation on $w$. Hence, cancellation of leftmost $s_i$ in the expression of $w^2$ is possible only with the other $s_i$ appearing in the expression of $w^2$ by repeated application of the flip transformation. This happens only if the leftmost $s_{i+1}$ occuring between the two $s_i$'s in the expression of $w^2$ cancel. But that is not possible since there is a $s_i$ between the two $s_{i+1}$'s. Thus,  $s_iw_1s_{i+1}w_2s_iw_1s_{i+1}w_2\neq 1$, a contradiction. The proof of the converse is immediate.
\end{proof}

As a consequence of Corollary \ref{cor1} and Proposition \ref{prop-inv}, we obtain the following result.

\begin{corollary}\label{cor3.2}
Let $w$ be an element of $T_n$. Then $w$ is an involution if and only if $w$ is conjugate to a cyclically reduced word of the form $s_{i_1} s_{i_2}\dots s_{i_k}$ such that $i_{t+1} - i_t\geq 2$. Furthermore, any two distinct cyclically reduced words of this form are not conjugates.
\end{corollary}

Note that a cyclically reduced word $w$ is an involution if and only if it can be written in the form $s_{i_1} s_{i_2}\dots s_{i_k}$ such that $i_{t+1} - i_t\geq 2$. Set
\begin{equation}\label{A_n}
\mathcal{A}_n = \big\{s_{i_1} s_{i_2}\dots s_{i_k} \mid 1\leq i_t\leq n-1, ~ i_{t+1} - i_t\geq 2 \big\}.
\end{equation}
The following result, whose  proof is immediate from the presentation of $T_n$, gives ranks of the centralisers of cyclically reduced involutions.

\begin{lemma}\label{centraliser}
Let $w =s_{i_1} s_{i_2}\dots s_{i_k}$ be an involution in $T_n$, where $i_{t+1} - i_t\geq 2$ for all $1 \le t \le k-1$. Then $\C_{T_n}(w)  = \big\langle S\setminus \bigcup\limits_{t=1}^{k}  s_{i_t}^*\big\rangle$, and consequently $rank \big(\C_{T_n}(w)\big) = (n-1)-  |\bigcup\limits_{t=1}^{k}  s_{i_t}^*|.$
\end{lemma}

We now present the main result of this section.

\begin{theorem}\label{theoremrho}
Let $\rho_n$ denote the number of conjugacy classes of involutions in $T_n$. Then $$\rho_n =1+ \rho_{n-1} + \rho_{n-2}$$
 for all $n\geq 4$,  where $\rho_2=1$ and $\rho_3=2$.
\end{theorem}

\begin{proof}
Consider the set $\mathcal{A}_n$ as defined in \eqref{A_n}. Then, by Corollary \ref{cor3.2}, we have $\rho_n = |\mathcal{A}_n|$. Note that $\mathcal{A}_2 = \{s_1\}$ and $\mathcal{A}_3 = \{s_1, s_2\}$, which implies that $\rho_2=1$ and $\rho_3=2$. We now proceed to compute $\rho_n$ for $n \geq 4$.  We define three mutually disjoint subsets of $\mathcal{A}_n$ as follows:
\begin{itemize}
\item[(i)] $\mathcal{B}_n=\{s_{n-1}\}$.
\item[(ii)] $\mathcal{C}_n=\{s_{i_1} s_{i_2}\dots s_{i_k} \mid  k>1, i_k=n-1\}$.
\item[(iii)] $\mathcal{D}_n=\{s_{i_1} s_{i_2}\cdots s_{i_k} \mid  i_k< n-1\}$.
\end{itemize} 
It is easy to see that $\mathcal{A}_n = \mathcal{B}_n  \sqcup \mathcal{C}_n  \sqcup  \mathcal{D}_n$, and hence
$$|\mathcal{A}_n| = |\mathcal{B}_n| + |\mathcal{C}_n| + |\mathcal{D}_n| = 1 + |\mathcal{C}_n| + |\mathcal{D}_n|.$$
Now, the map sending $s_{i_1} s_{i_2}\dots s_{i_k}$ to $s_{i_1} s_{i_2}\dots s_{i_{k-1}}$ gives a bijection between the sets $\mathcal{C}_n$ and $\mathcal{A}_{n-2}$,  and hence $| \mathcal{C}_n | = | \mathcal{A}_{n-2} |$. Also, note that $\mathcal{D}_n = \mathcal{A}_{n-1}$. Thus, we have 
$$|\mathcal{A}_n| = 1 + |\mathcal{A}_{n-1}| + |\mathcal{A}_{n-2}|,$$
which implies that 
$$\rho_n = 1 + \rho_{n-1} + \rho_{n-2}.$$
\end{proof}

\begin{corollary}
For each $n\geq 2$, $ \rho_{n} + 1 =  \textbf{F}_{n+1} $, where $\left(\textbf{F}_{n}\right)_{n \geq 1}$ is the well-known Fibonacci sequence with $\textbf{F}_1 =\textbf{F}_2 = 1$. In particular,
$$\rho_{n} = \sum_{k=1}^{\lfloor{\frac{n}{2}}\rfloor} 
{n-k \choose k}.$$
\end{corollary}

\begin{proof}
Observe that $\rho_n +1= (\rho_{n-1} +1)+ (\rho_{n-2}+1)$. The first assertion is clear from  Theorem \ref{theoremrho}. The formula for $\rho_n$ can be derived from the well-known value of $(n+1)$-term of the Fibonacci sequence \cite[Chapter 3, Section 3.1.2]{Rosen}.
\end{proof}
\bigskip

\section{z-classes in twin groups}\label{z-classes}
Two elements $x,y$ of a group $G$ are said to be $z$-\textit{equivalent} if their centralisers $\C_G(x)$ and $\C_G(y)$ are conjugates in $G$. A $z$-equivalence class is called a $z$-\textit{class}. We would like to mention that $z$-classes also appear naturally in geometry and topology. We refer the reader to \cite{Kulkarni} for a quick review of the same.
\par
It is clear that conjugate elements are $z$-equivalent and the converse is not true. For example, $\C_{T_4}(s_1s_3)=\C_{T_4}(s_3)$, but $s_1s_3$ and $s_3$ are not conjugate. Thus, to investigate $z$-classes in $T_n$, it is sufficient to study centralisers of cyclically reduced words (by Corollary \ref{cor1}). Note that every element of $T_n$ is either torsion-free or of order 2. We first show that only $T_2$ and $T_3$ have finitely many $z$-classes, and then compute number of $z$-classes of involutions in $T_n$ for $n \ge 2$.

\begin{proposition}
$T_n$ has finitely many $z$-classes if and only if $n= 2$ or $3$.
\end{proposition}

\begin{proof}
Since $T_2 \cong \mathbb{Z}_2$,  there are two conjugacy classes and only one $z$-class. In $T_3$, there are infinitely many conjugacy classes, namely, $s_1^{T_3}, s_2^{T_3}, (s_1s_2)^{T_3}, \big((s_1s_2)^2\big)^{T_3},  \big((s_1s_2)^3\big)^{T_3},$ and so on. We note that 
\begin{align*}
\C_{T_3}(s_1) &=\langle s_1 \rangle,\\
\C_{T_3}(s_2) &=\langle s_2 \rangle,\\
\C_{T_3}(s_1s_2) &=\langle s_1s_2 \rangle =  \C_{T_3}\big((s_1s_2)^m\big),~ m \geq 2.
\end{align*}
By Theorem \ref{conditionforconjugate}, it follows that $\C_{T_3}(s_1),\; \C_{T_3}(s_2)$ and $\C_{T_3}(s_1s_2)$ are pairwise not conjugate. Therefore, there are three $z$-classes in $T_3$.
\par
Now, we proceed to prove that $T_n$ has infinitely many $z$-classes for $n\geq 4$. It suffices to construct an infinite sequence of cyclically reduced words in $T_n$ such that their centralisers are not pairwise conjugate in $T_n$.
We define $X_1 = s_1s_2$, $X_2=s_1s_2s_3$, $X_3 = s_1s_2s_3s_2,$ $X_{2i} = X_{2i-1}s_3$, $X_{2i+1} = X_{2i}s_2$ for $i \geq 2$. It is easy to check that $ \C_{T_n}(X_1)=\langle X_1 \rangle\times H $ and $\C_{T_n}(X_j)= \langle X_j \rangle \times K$ for $ j\geq 2$, where $H = \langle s_4, s_5, \dots,  s_{n-1}\rangle$ and $K =\langle s_5, s_6, \dots , s_{n-1} \rangle$. It can be easily deduced that if $\C_{T_n}(X_i)$ is conjugate to $\C_{T_n}(X_j)$ for some $i\neq j$, then $X_i$ is conjugate to $X_j$. But this is not possible due to Theorem \ref{conditionforconjugate}.
\end{proof}

Now, we proceed to compute the number of $z$-classes of involutions in $T_n$. As mentioned earlier, it is sufficient to consider centralisers of cyclically reduced involutions in $T_n$. Thus, for the rest of this section, by an involution, we mean a cyclically reduced involution, that is, an element of $\mathcal{A}_n$. We begin with the following observation.

\begin{lemma}\label{equal-non-conjugate}
Let $w_1$ and $w_2$ be two involutions in $T_n$. Then either $\C_{T_n}(w_1) = \C_{T_n}(w_2)$ or $\C_{T_n}(w_1)$ and $\C_{T_n}(w_2)$ are not conjugates of each other.
\end{lemma}

\begin{proof}
Let us suppose $\C_{T_n}(w_1) \neq \C_{T_n}(w_2)$. Then, without loss of generality, we can assume that there exists some $s_j\in \C_{T_n}(w_1) \setminus \C_{T_n}(w_2)$. Thus, we can write $\C_{T_n}(w_2) = \langle s_{i_1}, s_{i_2}, \dots , s_{i_k} \rangle$ such that $j \notin \{ i_1, i_2, \dots , i_k\}$. Consequently, for each $g \in T_n$,  $\C_{T_n}(w_2)^g = \langle s_{i_1}^g, s_{i_2}^g, \dots , s_{i_k}^g \rangle$. Thus, each word in $\C_{T_n}(w_2)^g$ contains $s_j$ even number of times, and hence $s_j \notin \C_{T_n}(w_2)^g$ for any $g\in T_n$. Therefore,  $\C_{T_n}(w_1)$ and $\C_{T_n}(w_2)$ are not conjugates of each other.
\end{proof}

By virtue of the preceding lemma, the number of $z$-classes of involutions in $T_n$ is equal to the number of distinct centralisers of cyclically reduced involutions in $T_n$.

Let $\lambda_n$ denote the number of distinct centralisers of involutions in $T_n$, $n\geq 2$. A direct computation yields $\lambda_2 =1$,  $\lambda_3 =2$,  $\lambda_4 =2$, $\lambda_5 =5$ and $\lambda_6 =8$. The following main result of this section gives a recursive formula for $\lambda_n$, $n\geq 7$.

\begin{theorem}\label{thm-zclass}
Let $\lambda_n$ be as defined above. Then,  for $n\geq 7$,  $$\lambda_n = \Big( \sum_{i=3}^{n-2} \lambda_i \Big )- \lambda_{n-4} + n-2 .$$
\end{theorem}

We establish the preceding theorem through a sequence of lemmas.

\begin{lemma}\label{Lemma-newCounting2}
The number of distinct centralisers of involutions ending with $s_i$ in $T_n$ is equal to number of distinct centralisers of involutions ending with $s_i$ in $T_m$ for all $n, m \geq i+1$.
\end{lemma}
\begin{proof}
It is sufficient to prove the assertion for $n = i+1$ and $m> n$. Let $ws_i$ be an involution ending with $s_i$. Then

\begin{equation*}
\C_{T_m}(ws_i) =
 \begin{cases}
\C_{T_n}(ws_i)~~ \text{ if } m=n+1 =i+2,\\
 \big\langle \C_{T_n}(ws_i), s_{n+1}, s_{n+2}, \dots, s_{m-1} \big\rangle~~ \text{ if }  m\geq n+2 = i+3.
 \end{cases}
 \end{equation*}
Hence $\C_{T_m}(w_1s_i) = \C_{T_m}(w_2s_i)$ if and only if $\C_{T_n}(w_1s_i) = \C_{T_n}(w_2s_i)$. This completes the proof.
\end{proof}
 
The preceding lemma allows us to define $\alpha_i$ as the number of distinct centralisers of involutions ending with $s_i$ in $T_n$ for all $n\geq i+1$. 

\begin{lemma}\label{Lemmafor-newLamdaAlpha}
In $T_n$, the centraliser of an involution ending with $s_i$ is not equal to the centraliser of any involution ending with $s_j$ for $i< j$, unless $i=n-3$ and $j=n-1$. Moreover, the centraliser of an involution ending with $s_{n-3}$ is equal to the centraliser of some involution ending with $s_{n-1}$.
\end{lemma}

\begin{proof}
Let $w_1s_i$ and $w_2s_j$ be two involutions ending with $s_i$ and $s_j$, respectively, such that $i< j$. If $j\leq n-2$, then $s_{j+1}\in \C_{T_n}(w_1s_i)$, but $s_{j+1}\notin \C_{T_n}(w_2s_j)$. If $j = n-1$, then unless $i=n-3$, $s_{n-2}\in  \C_{T_n}(w_1s_i)$, but $s_{n-2}\notin \C_{T_n}(w_2s_j)$. This proves the first assertion of the lemma. For the second assertion, if $ws_{n-3}$ is an involution ending with $s_{n-3}$, then $\C_{T_n}(ws_{n-3}) = \C_{T_n}(ws_{n-3}s_{n-1})$.
\end{proof}

\begin{lemma}\label{lemmafor-newlambda}
For $n\geq 4$,
$$\lambda_n = \Big( \sum_{i=1}^{n-1} \alpha_i \Big ) - \alpha_{n-3}.$$
\end{lemma}

\begin{proof}
From the preceding lemma, we see that
\begin{eqnarray*}
\lambda_n &=& \sum_{i=1}^{n-1} \big (\textit{number of distinct centralisers of involutions ending with}~ s_i \big )\\
& & - \textit{number of distinct centralisers of involutions ending with}~ s_{n-3}\\
 &= & \Big ( \sum_{i=1}^{n-1} \alpha_i \Big ) - \alpha_{n-3},
\end{eqnarray*}
which is desired.
\end{proof}

\begin{lemma}\label{Lemmafor-newLamdaAlpha}
In $T_n$, the centraliser of an involution ending with $s_is_{n-1}$ is not equal to the centraliser of any involution ending with $s_js_{n-1}$ for $1\leq i< j\leq n-3$, unless $i=n-5$ and $j=n-3$. Moreover, the centraliser of an involution ending with $s_{n-5}s_{n-1}$ is equal to the centraliser of some involution ending with $s_{n-3}s_{n-1}$.
\end{lemma}

\begin{proof}
Let $w_1s_is_{n-1}$ and $w_2s_js_{n-1}$ be two involutions ending with $s_is_{n-1}$ and $s_js_{n-1}$, respectively, with $1\leq i< j\leq n-3$. If $j\leq n-4$, then $s_{j+1}\in \C_{T_n}(w_1s_is_{n-1})$, but $s_{j+1}\notin \C_{T_n}(w_2s_js_{n-1})$. If $j = n-3$, then unless $i=n-5$, $s_{n-4}\in  \C_{T_n}(w_1s_is_{n-1})$, but $s_{n-4}\notin \C_{T_n}(w_2s_js_{n-1})$. This proves the first part of the lemma. For the second assertion, if $ws_{n-5}s_{n-1}$ is an involution ending with $s_{n-5}s_{n-1}$, then $\C_{T_n}(ws_{n-5}s_{n-1}) = \C_{T_n}(ws_{n-5}s_{n-3}s_{n-1})$.
\end{proof}

\begin{lemma}\label{LemmaCounting-new1}
For all $i\leq n-3$, the number of distinct centralisers of involutions ending with $s_is_{n-1}$ is equal to the number of distinct centralisers of involutions ending with $s_i$ in $T_n$.
\end{lemma}
\begin{proof}
Note that $\C_{T_n}(w{s_{n-3}}) = \C_{T_n}(w{s_{n-3}}s_{n-1})$. But for $i\leq n-4$, we have $s_{n-2} \notin \C_{T_n}(w{s_i}s_{n-1})$ and $\C_{T_n}({ws_i}) = \big\langle \C_{T_n}(ws_is_{n-1}), s_{n-2} \big\rangle$ . Thus, $\C_{T_n}(w_1s_is_{n-1}) = \C_{T_n}(w_2s_is_{n-1})$ if and only if $\C_{T_n}(w_1s_i) = \C_{T_n}(w_2s_i)$.
 \end{proof}

\begin{lemma}\label{lemmafor-newalpha} For $n\geq 5$,
$$ \alpha_{n-1}= 1 + \Big( \sum\limits_{i=1}^{n-3} \alpha_i \Big) -  \alpha_{n-5}.$$
\end{lemma}

\begin{proof}
The set of centralisers of involutions ending with $s_{n-1}$ in $T_n$ can be divided into two disjoint subsets, namely, $\{ \C_{T_n}(s_{n-1})\}$ and  the set of centralisers of involutions ending with $s_{n-1}$ and of length strictly greater than 1. The proof now follows from lemmas \ref{Lemmafor-newLamdaAlpha} and \ref{LemmaCounting-new1}.
\end{proof}

\bigskip
\textit{Proof of Theorem \ref{thm-zclass}.}
Replacing $n$ by $n+2$ in the preceding result and using Lemma \ref{lemmafor-newlambda}, we get $\alpha_{n+1}= 1+ \lambda_n$ for $n \ge 3$. A repeated use of this identity in Lemma \ref{lemmafor-newlambda} and some simplifications yields $$\lambda_n = \Big( \sum_{i=3}^{n-2} \lambda_i \Big )- \lambda_{n-4} + n-2$$ 
for $n \ge 7$, which is the desired formula. \hfill $\Box$
\bigskip

\section{Automorphisms of Twin groups}\label{Automorphism}
Using the preceding setup, we compute automorphisms of twin groups in full generality. Note that, the automorphism group of $T_3 \cong \mathbb{Z}_2 *\mathbb{Z}_2$ is well-known, and structure of $\Aut(T_n)$ is determined in \cite{James} for $n \geq 4$. However, our approach is elementary and yields an alternate proof for all $n \ge3$. Further, as applications, we show that $PT_n$ is not characteristic in $T_n$ and that $T_n$ is isomorphic to subgroup of $\Aut(PT_n)$ for $n \geq4$. We also determine the group of $\IA$ and normal automorphisms of $T_n$.

\begin{theorem}\label{thm4.1}
Let $T_n$ be the twin group with $n\geq 3$. Then the following hold:
\begin{enumerate}
\item  $\Aut(T_3)\cong T_3\rtimes \mathbb{Z}_2$.
\item  $\Aut(T_4)\cong T_4\rtimes S_3$.
\item $\Aut(T_n)\cong T_n\rtimes D_8\; \mbox{for}\; n \geq 5$, where $D_8$ is the dihedral group of order 8.
\end{enumerate}
\end{theorem}

We prove Theorem \ref{thm4.1} in two parts. First, we show that any automorphism that preserves conjugacy classes of generators is an inner automorphism. It is well-known \cite[Corollary 1]{bardakov} that the center $\Z(T_n) = 1$, and hence $\Inn(T_n)\cong T_n$ for $n>2$. We  then determine all the non-inner automorphisms of $T_n$, and show that $\Out(T_3)\cong \mathbb{Z}_2$, $\Out(T_4)\cong S_3$ and  $\Out(T_n) \cong D_8$ for $n\geq 5$. 

The following result characterises inner automorphisms of $T_n$. 

\begin{proposition}\label{inner}
Let $\phi$ be an automorphism of $T_n$ for $n\geq 3$. Then $\phi$ is inner if and only if $\phi(s_i)\in s_i^{T_n}$ for all $1\leq i\leq n-1$.
\end{proposition}

\begin{proof}
The forward implication is obvious. For the converse, suppose that $\phi(s_i)\in s_i^{T_n}$ for all $1\leq i\leq n-1$. We complete the proof in the following steps:\\

\noindent {\bf Step 1.} {\it There exists some $u\in T_n$ such that  $\widehat{u} \phi(s_{2i-1})= s_{2i-1}$ for all $1\leq i\leq \lfloor n/2 \rfloor$. }\\

 We begin by setting $\phi_1:=\phi$. Without loss of generality, we may assume that $\phi_1(s_1)=s_1$. Let us suppose that $\phi_1(s_3) = w_3^{-1}s_3w_3$, where $w_3$ is a reduced word. We claim that $w_3$ does not contain $s_2$. Let us, on the contrary, suppose that $w_3$ contains $s_2$. Then $s_1$ does not commute with $w_3^{-1}s_3w_3$, but $s_1$ commutes with $s_3$. This is a  contradiction to the fact that automorphisms preserve commuting relations. Thus, our claim is true. Next, we define $\phi_3:= \widehat{w_3}^{-1}\phi_1$. Note that  $\phi_3(s_1) = s_1$ and $\phi_3(s_3) = s_3$. 

Let us now suppose that $\phi_3(s_5) = w_5^{-1}s_5w_5$, where $w_5$ is a reduced word. Suppose that the word $w_5$  contains $s_2$ or $s_4$ or both. Then $s_3$ and $s_5$ commute but their images do not commute under the automorphism $\phi_3$, leading to a contradiction. Hence, $w_5$ contains neither $s_2$ nor $s_4$. Now, we define $\phi_5= \widehat{w_5}^{-1} \phi_3$.  Note that  $\phi_5(s_1) = s_1$, $\phi_5(s_3) = s_3$ and $\phi_5(s_5)=s_5$. 
 
Again, suppose $\phi_5(s_7) = w_7^{-1}s_7w_7$, where $w_7$ is a reduced word. Repeating the argument, we can show that $w_7$ does not contain $s_2$, $s_4$ and $s_6$. Define $\phi_7:= \widehat{w_7}^{-1} \phi_5$. Note that  $\phi_7(s_1) = s_1$, $\phi_7(s_3) = s_3$, $\phi_7(s_5)=s_5$  and $\phi_7(s_7)=s_7$. Continuing this process, we finally get $\phi_{2k-1}(s_{2i-1})=s_{2i-1}$, for all $1\leq i\leq \lfloor n/2 \rfloor$ and $k=\lfloor n/2 \rfloor$. This completes the proof of Step 1.\\

\noindent {\bf Step 2.} {\it There exists some $v\in T_n$ such that  $\widehat{v} \phi(s_{2i})= s_{2i}$ for all $1\leq i\leq \lfloor n-1/2 \rfloor$. }\\

The proof of this step goes along the same lines as that of Step 1.\\

\noindent {\bf Step 3.} {\it Without loss of generality, we can assume that there exists a reduced word $w\in T_n$ such that  $ \phi(s_{2i-1})= s_{2i-1}$ for all $1\leq i\leq \lfloor n/2 \rfloor$ and $\phi(s_{2i})= w^{-1}s_{2i}w$ for all $1\leq i\leq \lfloor n-1/2 \rfloor$. }\\

\noindent This follows immediately from steps $1$ and $2$.\\

\noindent \textbf{Step 4.} {\it If $w$ is the reduced word as in Step 3, then $\phi$ is an inner automorphism induced by some subword of $w$.}\\

We write $w= s_{i_1} s_{i_2}\dots s_{i_k}$. Note that, if $i_1$ is even, then $w^{-1}s_{2i}w$ = $w'^{-1}s_{2i}w'$ for all $1\leq i\leq \lfloor n-1/2 \rfloor$, where $w'= s_{i_2}s_{i_3}\dots s_{i_k}$. On the other hand, if $i_k$ is odd, then $ \widehat{s_{i_k}}^{-1}\phi(s_{2i-1})= s_{2i-1}$ for all $1\leq i\leq \lfloor n/2 \rfloor$ and $\widehat{s_{i_k}}^{-1}\phi(s_{2i})= w''^{-1}s_{2i}w''$ for all $1\leq i\leq \lfloor n-1/2 \rfloor$, where $w''= s_{i_1} s_{i_2}\dots s_{i_k-1}$. It follows that if $s_{i_1}, s_{i_2},\dots, s_{i_k}$ are all even indexed, then $\phi$ is the identity automorphism. Similarly, if $s_{i_1}, s_{i_2},\dots, s_{i_k}$ are all odd indexed, then $\phi$ is the inner automorphism induced by $w$. Further, if by applying finitely many flip transformations, we can write $w=w_1w_2$, where $w_1$ is subword with even indexed generators and $w_2$ a subword  with odd indexed generators,  then $\phi$ is the inner automorphism induced by $w_2$. \par

Now suppose that  $i_1$ is odd, $i_k$ is even and that we cannot bring an even indexed generator to the leftmost position and an odd indexed generator to rightmost position in the expression of $w$ by finitely many flip transformations on $w$. We would derive a contradiction by proving that $\phi$ is not surjective in this case.
\par

We note that $s_{i_k}\neq \phi(s_j)$ for all $j\neq i_k$. Suppose that $s_{i_k}= \phi(s_{i_k})=s_{i_k}s_{i_k-1}\dots s_{i_1}s_{i_k}s_{i_1} s_{i_2}\dots s_{i_k}$. This implies $s_{i_k-1}\dots s_{i_1}s_{i_k}s_{i_1} s_{i_2}\dots s_{i_k} = 1$. Thus, every generator (in particular $s_{i_k}$) appearing in the expression $s_{i_k-1}\dots s_{i_1}s_{i_k}s_{i_1} s_{i_2}\dots s_{i_k}$ should get cancelled by some elementary transformation. But as $w=s_{i_1} s_{i_2}\dots s_{i_k}$ is a reduced word, deletion of $s_{i_k}$ is possible only if we can bring the $s_{i_k}$ to the leftmost position in the expression of $w$ by some flip transformations. But this is not possible, and hence $s_{i_k}\neq \phi(s_{i_k})$.
\par
Now suppose that $s_{i_k}= \phi(x)$ for some reduced word $x=s_{j_1}s_{j_2}\dots s_{j_t}$ of length greater than $1$, i.e., $t> 1$. There are four possibilities on the choice of indices of $s_{j_1}$ and $s_{j_t}$ to be even or odd. Here, we consider one case, i.e. $j_1$ is odd and $j_t$ is even. Rest of the cases follow similarly.
\par
Now, we can write $x=x_1 x_2 \dots x_{2l}$, where $2\leq 2l\leq t$, the odd indexed subwords (i.e. $x_1$, $x_3, \dots, x_{2l-1}$)  contain generators of odd index ($s_1,s_3, \dots,$ etc.) and even indexed subwords (i.e. $x_2$, $x_4, \dots, x_{2l}$)  contain generators of even index ($s_2,s_4, \dots,$ etc.). We have
$$s_{i_k}=\phi(x_1 x_2\dots x_{2l})=x_1 (w^{-1}x_2w) x_3(w^{-1}x_4w)\dots x_{2l-1}(w^{-1}x_{2l}w).$$
\par
Note that no deletion is possible in the expression $x_1 (w^{-1}x_2w) x_3(w^{-1}x_4w)\dots x_{2l-1}(w^{-1}x_{2l}w)$, because of the assumption that $i_1$ is odd, $i_k$ is even and that we cannot bring an even indexed generator to the leftmost position and an odd indexed generator to rightmost position in the expression of $w$ by finitely many flip transformations on $w$. Thus, $$\ell(s_{i_k})=1<2\leq \ell \big(x_1 (w^{-1}x_2w) x_3(w^{-1}x_4w)\dots x_{2l-1}(w^{-1}x_{2l}w)\big),$$ and hence $s_{i_k}\neq \phi(x_1 x_2\dots x_{2l})$ showing that $\phi$ is not surjective. This completes the proof of the proposition.
\end{proof}

\begin{proposition}\label{flip}
The map $\psi:T_n \to T_n$ given by $\psi(s_i)=s_{n-i}$, $1\leq i\leq n-1$, extends to an order 2  non-inner automorphism of $T_n$ for all $n \geq 3$.
\end{proposition}
\begin{proof}
The proof follows from the definition of $\psi$.
\end{proof}

\begin{proposition}\label{tfourauto}
The following hold in $T_4$:
\begin{itemize}
\item[(i)]The map $\tau : T_4 \to T_4$ given by $\tau(s_1)=s_1s_3$, $\tau(s_2)=s_2$ and $\tau(s_3)=s_1$, extends to an order 3 non-inner automorphism of $T_4$.
\item[(ii)] The subgroup generated by $\tau$ and $\psi $ is isomorphic to $S_3$.
\end{itemize}
\end{proposition}
\begin{proof}
That $\tau$ is a  non-inner  automorphism of order 3 is obvious. Since $\tau$ satisfies the relation $$\psi \tau \psi = {\tau}^2,$$ we have $\langle \psi, \tau \rangle \cong S_3$.
\end{proof}

\begin{proposition}\label{dihedral}
The following hold in $T_n$ for $n\geq 5$:
\begin{itemize}
\item[(i)]The map $\kappa : T_n \to T_n$ given by $\kappa(s_3)=s_{n-3}s_{n-1}$ and $\kappa(s_i)=s_{n-i}$ for $i\neq 3$ extends to an order 4 non-inner  automorphism of $T_n$.
\item[(ii)] The subgroup generated by $\kappa$ and $\psi $ is isomorphic to $D_8$.
\end{itemize}
\end{proposition}
\begin{proof}
It is easy to check that $\kappa$ extends to a non-inner  automorphism of order 4. Since $\kappa$ satisfies the relation $$\psi \kappa \psi = \kappa^{-1},$$ it follows that $\langle \psi, \kappa \rangle \cong D_8$.

\end{proof}

\begin{lemma}\label{lem4.01}
Let $\phi$ be an automorphism of $T_4$. Then $\phi(s_1), \phi(s_3)\in s_1^{T_4}, s_3^{T_4}$ or $(s_1s_3)^{T_4}$ and $\phi(s_2) \in s_2^{T_4}$.
\end{lemma}

\begin{proof}
It follows from Corollary \ref{centraliser} that $s_1$, $s_3$ and $s_1s_3$ are the only involutions (upto conjugation) with centralisers of  rank two and $s_2$ is the only involution (upto conjugation) with centraliser of  rank one. The result follows since their images under any automorphism should again be involutions with centralisers of the same rank.
\end{proof}

\begin{lemma}\label{lem4.1}
Let $\phi$ be an automorphism of $T_n$ for $n\geq 3$ and $n \neq4$. Then either $\phi(s_1)\in s_1^{T_n}$ and $\phi(s_{n-1})\in s_{n-1}^{T_n}$ or $\psi \phi(s_1)\in s_1^{T_n}$ and $\psi \phi(s_{n-1})\in s_{n-1}^{T_n}$.
\end{lemma}

\begin{proof}
It follows from Corollary \ref{centraliser} that $s_1$ and $s_{n-1}$ are the only involutions with centralisers (upto conjugation) of  rank $n-2$ in $T_n$. The result follows since  their images under any automorphism should again be involutions with centralisers of the same rank $n-2$.
\end{proof}

\begin{lemma}\label{lem4.3}
Let $n\geq 5$ and $\phi \in \Aut(T_n)$. Then for all $2\leq i\leq n-2$, $\phi(s_i) \in s_j^{T_n}$ for some $2\leq j\leq n-2$ or $\phi(s_i) \in (s_1s_3)^{T_n}$ or $(s_{n-3}s_{n-1})^{T_n}$.
\end{lemma}

\begin{proof}
Fix an $i$ such that $2\leq i \leq n-2$. We observe that $s_i$ is an involution and its centraliser has rank $n-3$. From Corollary \ref{centraliser}, it is clear that only $s_2, s_3,\dots, s_{n-2}$ and $s_1s_{n-1}$, $s_1s_3$ and $s_{n-3}s_{n-1}$  are cyclically reduced involutions whose centralisers have rank $n-3$. Further, from Lemma \ref{lem4.1}, it follows that $\phi(s_i)\notin (s_1s_{n-1})^{T_n}$ for $2\leq i \leq n-2$.
\end{proof}

\begin{lemma}\label{lem4.4}
Let $n\geq 5$ and $\phi \in \Aut(T_n)$ be an automorphism such that $\phi(s_1)\in s_1^{T_n}$. Then the following hold:
\begin{enumerate}
\item[(i)] $\phi(s_i)\in s_i^{T_n}$ for all $2\leq i\leq n-2$ and $3\neq i \neq n-3$.
\item[(ii)]  $\phi(s_3)\in s_3^{T_n}$ or $(s_1s_3)^{T_n}$.
\item[(iii)] $\phi(s_{n-3})\in s_{n-3}^{T_n}$ or $(s_{n-3}s_{n-1})^{T_n}$.
\end{enumerate}
\end{lemma}

\begin{proof}
Suppose $\phi (s_2)\in g^{T_n}$ for some $g\in \{s_2,s_3,\ldots, s_{n-2}, s_1s_3, s_{n-3}s_{n-1}\}$. Choosing an appropriate inner automorphism say $\hat{w}$ and a reduced word $w'$, we get $\hat{w}(\phi (s_2))=g$ and $\hat{w}\big(\phi (w'^{-1}s_1w')\big)=s_1$. We note that $w'^{-1}s_1w'$ and $s_2$ do not commute. Since automorphisms preserve commuting relations, $s_1$ and $g$ also should not commute, and hence $g=s_2$. The proof can now be completed by repeating the argument.
\end{proof}

\indent {\bf Proof of Theorem \ref{thm4.1}}
Recall from propositions \ref{tfourauto} and \ref{dihedral} that $\langle \psi, \tau \rangle\cong S_3 $ and $\langle \psi, \kappa \rangle\cong D_8$. We observe that $\Inn(T_3) \cap \langle \psi \rangle$, $\Inn(T_4)\cap \langle \psi, \tau \rangle$ and $\Inn(T_n)\cap \langle \psi, \kappa \rangle$, $n \ge 5$, are all trivial. Thus, $\Inn(T_3) \rtimes \langle \psi \rangle \le \Aut(T_3)$, $\Inn(T_4) \rtimes \langle \psi, \tau \rangle \le \Aut(T_4)$ and $\Inn(T_n) \rtimes\langle \psi, \kappa \rangle \le \Aut(T_n)$ for $n\ge 5$.  It now remains to prove the reverse inclusions. Let $\phi$ be an automorphism of $T_n$. It follows from Proposition \ref{inner} and lemmas  \ref{lem4.01}, \ref{lem4.1}, \ref{lem4.3}, \ref{lem4.4}  that
\begin{enumerate}
\item[(a)] $\psi^{t}\phi \in \Inn(T_3)$ for some $0\leq t\leq 1$,
\item[(b)] $\psi^{m_1}\tau^{m_2}\phi \in \Inn(T_4)$ for some $0\leq m_1\leq 1$ and $0\leq m_2\leq 2$,
\item[(c)] $\psi^{n_1}\kappa^{n_2}\phi \in \Inn(T_n)$ for some $0\leq n_1\leq 1$ and $0\leq n_2\leq 3$, where $n \ge 5$.
\end{enumerate}
This completes the proof of the theorem. \hfill $\Box$
\par

\begin{corollary}\label{T_n}
The following hold in $T_n$:
\begin{enumerate}
\item[(i)] $\Out(T_3) \cong  \mathbb{Z}_2 \cong \langle \psi \rangle.$
\item[(ii)] $\Out(T_4)\cong S_3 \cong \langle\psi, \tau \rangle $.
\item[(iii)] $\Out(T_n)\cong D_8 \cong  \langle \psi, \kappa \rangle $ for $n\ge 5$.
\end{enumerate}
\end{corollary}

A consequence of our preceding analysis is the following result.

\begin{proposition}\label{not-characteristic}
$PT_n$ is characteristic in $T_n$ if and only if $n= 2,3$.
\end{proposition}
\begin{proof}
$PT_2$ being trivial is obviously characteristic in $T_2$. We observe that $PT_n$ is invariant under $\psi$. This follows since  the set $\big\{((s_is_{i+1})^3)^g~|~1 \le i \le n-2, ~g \in T_n\big\}$ generates $PT_n$ (\cite[Theorem 4]{bardakov}) and $\psi\big(((s_is_{i+1})^3)^g\big)= \big((s_{n-i}s_{n-i-1})^3\big)^{\psi(g)} \in PT_n$. This together with Theorem \ref{thm4.1}(1) implies that $PT_3$ is characteristic in $T_3$.
\par
For the reverse implication, first consider the element $(s_1s_2)^3 \in PT_4$. Then $\tau \big ( (s_1s_2)^3\big ) = (s_1s_3s_2)^3 \notin PT_4$ (since $\pi \big((s_1s_3s_2)^3\big) \neq 1$), and hence $PT_4$ is not invariant under $\tau$. Similarly, $\kappa \big((s_2s_3)^3\big)=(s_{n-2}s_{n-3}s_{n-1})^3 \notin PT_n$ for $n \ge 5$, and we are done.
\end{proof}

Since $PT_n$ is normal in $T_n$, there is a natural homomorphism 
$$\phi_n : T_n\cong \Inn(T_n)\rightarrow \Aut(PT_n),$$
obtained by restricting the inner automorphisms. It is proved in \cite{bardakov} that $\Ker(\phi_3)\neq 1$ and $\phi_4$ is injective. We show that this is the case for all $n \geq 4$. 

\begin{proposition}\label{EmbeddingTheorem} 
The homomorphism $\phi_n : T_n\rightarrow \Aut(PT_n)$ is injective if and only if $n\geq 4$.
\end{proposition}

\begin{proof}
Note that $\Ker(\phi_n)= \C_{T_n}(PT_n)$. It is easy to check that 
$$\C_{T_n}\big((s_is_{i+1})^3\big)=\big\langle s_1,s_2,\dots, s_{i-2},s_is_{i+1},s_{i+3}, s_{i+4}, \dots, s_{n-1} \big\rangle,$$
and  $$\C_{T_n}(PT_n)\leq \bigcap_{i=1}^{n-2} \C_{T_n}\big((s_is_{i+1})^3\big)=1.$$
\end{proof}

An $\IA$ automorphism of a group is an automorphism that acts as identity on the abelianization of the group. Note that inner automorphisms are $\IA$ automorphisms. It is easy to check that non-inner automorphisms of $T_n$ for $n\geq 3$ are not $\IA$ automorphisms. Therefore, we have the following result.

\begin{proposition}
Every  $\IA$ automorphism of $T_n$ is inner for $n \geq 3$.
\end{proposition}

An automorphism of a group is said to be \textit{normal} if it maps every normal subgroup onto itself. The following is an analogue of a similar result of Neshchadim for braid groups \cite{Neshchadim}.

\begin{proposition}
Every normal automorphism of $T_n$ is inner for $n \geq 3$.
\end{proposition}

\begin{proof}
Note that every inner automorphism is a normal automorphism. Thus, in view of Theorem \ref{thm4.1}, it suffices to prove that no automorphism in the sets $\{\psi\}$, $\{ \psi , \tau , \tau^2 , \tau\psi , \tau^2\psi \}$ and $\{\psi, \kappa,  \kappa^2, \kappa^3, \kappa\psi,  \kappa^2\psi, \kappa^3\psi \}$ is normal for $n=3$, $n=4$ and $n \ge 5$, respectively. 

We first prove that $\psi$ is not a normal automorphism of $T_n$ for all $n\geq 3$. Take $N$ to be the normal closure of $s_1$ in $T_n$. Note that for each element $g\in N$ and each generator $s_i$, $i\geq 2$, number of $s_i$'s present in the expression of $g$ is even. This implies that $s_1\in N$ whereas  $\psi(s_1)=s_{n-1} \not\in N$, and hence $\psi$ is not normal.
\par
It follows from the proof of Proposition \ref{not-characteristic} that $PT_4$ is not invariant under $\tau$, and so under its inverse $\tau^2$. Hence, both $\tau$ and $\tau^2$ are not normal. Similarly, by Proposition \ref{not-characteristic}, it follows that $\kappa$ and its inverse  $\kappa^3$ are not normal. Further,  $\kappa^2$ is not normal, since $(s_2s_3)^3 \in PT_n$ whereas $\kappa^2 \big((s_2s_3)^3\big)=(s_2s_4s_3s_1)^3$ for $n=5$, $\kappa^2 \big((s_2s_3)^3\big)=(s_2s_3s_5s_1)^3$ for $n=6$ and $\kappa^2 \big((s_2s_3)^3\big)=(s_2s_3s_1)^3$ for $n \ge 7$. In each of these cases, $\kappa^2 \big((s_2s_3)^3\big) \not\in PT_n$.
\par
For the remaining cases, we recall that $PT_n$ is invariant under $\psi$. Consequently, if  $PT_4$ is invariant under $\tau^i\psi$, then it is so under $\tau^i$, a contradiction. Similarly, if $PT_n$, $n \ge 5$, is invariant under $\kappa^j\psi$, then it is so under $\kappa^j$, again a contradiction. Thus, the only normal automorphisms of $T_n$ are the inner automorphisms.
\end{proof}
\bigskip
 
 \section{Representations of twin groups by automorphisms}\label{representations}
 It follows from Proposition \ref{EmbeddingTheorem} that for $n=4, 5, 6$ we have faithful representations $$T_4 \hookrightarrow \Aut(F_7),$$
$$T_5 \hookrightarrow \Aut(F_{31}),$$
and 
$$T_6 \hookrightarrow \Aut \big(F_{71} * ( *_{20}(\mathbb{Z} \oplus \mathbb{Z})) \big).$$
It is a natural question whether there exists a (faithful) representation of $T_n$ into $\Aut(F_n)$ analogous to the classical Artin representation of the braid group. We conclude with the following result. 

 \begin{theorem}\label{faithful-representation}
 The map $\mu_n: T_n \to \Aut(F_n)$ defined by the action of generators of $T_n$ by
 \[ \mu_n(s_i) :
 \begin{cases}
 x_i \mapsto x_ix_{i+1},\\
 x_{i+1} \mapsto x_{i+1}^{-1},\\
 x_j \mapsto x_j, \quad j \neq i, i+1,\\
 \end{cases}
 \]
 is a representation of $T_n$. Moreover, $\mu_n$ is faithful if and only if $n=2,3$.
 \end{theorem}
 
 \begin{proof}
 We begin by proving that $\mu_n$ is a representation. Clearly, $s_i^2$ act as identity automorphism of $F_n$. Moreover, the action of $ \mu_n(s_i)\mu_n(s_j)$ and $\mu_n(s_j)\mu_n(s_i)$ on generators $x_1, \dots , x_n$ is 
 \begin{equation*}
 \mu_n(s_i)\mu_n(s_j) :
 \begin{cases}
 x_i \mapsto x_ix_{i+1},\\
 x_{i+1} \mapsto x_{i+1}^{-1},\\
 x_j \mapsto x_jx_{j+1},\\
 x_{j+1} \mapsto x_{j+1}^{-1},\\
 \end{cases}
 \quad
 \mu_n(s_j)\mu_n(s_i):
 \begin{cases}
 x_i \mapsto x_ix_{i+1},\\
 x_{i+1} \mapsto x_{i+1}^{-1},\\
 x_j \mapsto x_jx_{j+1},\\
 x_{j+1} \mapsto x_{j+1}^{-1},\\
 \end{cases}
 \end{equation*}
 for all $| i - j | \geq 2.$
 \par
 
Faithfulness for $n=2$ is obvious. For the case $n=3$, we know that an arbitrary element of $T_3$ is of the form $(s_1s_2)^m$ or ${(s_1s_2)^m}s_1$ or $s_2(s_1s_2)^m$ for some integer $m$. If $\Ker(\mu_3) \neq 1$, then there exists a non-trivial element $s \in T_3$ such that $\mu_3(s)(x_i)= x_i$ for $i=1,2,3$. We show that no such element exists. We first consider elements of the form $(s_1s_2)^m$. It is easy to see that the action of all odd powers of $s_1s_2$ gives a non-identity automorphism of $F_3$, since it sends $x_3$ to $x_3^{-1}$. On the other hand, for even powers of $s_1s_2$, we have 
\[\mu_3 \big((s_1s_2)^{2k}\big):
\begin{cases}
x_1 \mapsto x_1x_3^k,\\
x_2 \mapsto x_3^{-k}x_2x_3^{-k},\\
x_3 \mapsto x_3,
\end{cases}
\] 
for all integer $k$.  Next we consider ${(s_1s_2)^m}s_1$. Again, if $m$ is odd, the action is non-trivial and if $m=2k$, then we have 
\[\mu_3 \big({(s_1s_2)^{2k}}s_1\big):
\begin{cases}
x_1 \mapsto x_1x_2x_3^{-k},\\
x_2 \mapsto x_3^{k}x_2^{-1}x_3^{k},\\
x_3 \mapsto x_3.
\end{cases}
\] 
Similarly, for $s_2(s_1s_2)^m$, we have a non-trivial action if $m$ is even. If $m=2k-1$, then we have 
\[\mu_3 \big({s_2(s_1s_2)^{2k-1}}\big):
\begin{cases}
x_1 \mapsto x_1x_2x_3^{k},\\
x_2 \mapsto x_3^{-k}x_2^{-1}x_3^{-k},\\
x_3 \mapsto x_3.
\end{cases}
\] 
Thus,  $\mu_3 : T_3 \to \Aut(F_3)$ is faithful.
\par
Finally, we show that  $\mu_n$ is not faithful for $n \geq 4$. Consider the element $$x=(s_2s_3)^{-2} s_1 (s_2s_3)^2 s_1 (s_2s_3)^{2} s_1 (s_2s_3)^{-2} s_1 \in T_n,~ n \geq 4.$$ Since $\pi(x) \neq 1$, it follows that $x
\neq 1$. It is easy to check that 
\begin{equation*}\label{action1}
\mu_n \big({(s_2s_3)^{2}}\big):
\begin{cases}
x_1 \mapsto x_1,\\
x_2 \mapsto x_2x_4,\\
x_3 \mapsto x_4^{-1}x_3x_4^{-1},\\
x_j \mapsto x_j~,~j \geq 4,
\end{cases}
\end{equation*}
\begin{equation*}\label{action2}
\mu_n \big({(s_2s_3)^{-2}}\big):
\begin{cases}
x_1 \mapsto x_1,\\
x_2 \mapsto x_2x_4^{-1},\\
x_3 \mapsto x_4x_3x_4,\\
x_j \mapsto x_j~,~j \geq 4,
\end{cases}
\end{equation*}

\begin{equation}\label{action3}
\mu_n \big({(s_2s_3)^{2}s_1(s_2s_3)^{-2}}\big):
\begin{cases}
x_1 \mapsto x_1x_2x_4,\\
x_2 \mapsto x_4^{-1}x_2^{-1}x_4^{-1},\\
x_j \mapsto x_j~,~j \geq 3,
\end{cases}
\end{equation}
and
\begin{equation}\label{action4}
\mu_n \big({(s_2s_3)^{-2}s_1(s_2s_3)^{2}}\big):
\begin{cases}
x_1 \mapsto x_1x_2x_4^{-1},\\
x_2 \mapsto x_4x_2^{-1}x_4,\\
x_j \mapsto x_j~,~j \geq 3.
\end{cases}
\end{equation}
Using \ref{action3}, \ref{action4} and action of $s_1$, we conclude that $x \in \Ker (\mu_n)$.
\end{proof}

\noindent\textbf{Acknowledgments.}
The authors are grateful to Valeriy Bardakov for his interest in this work and for his many useful comments. Tushar Kanta Naik and Neha Nanda thank IISER Mohali for the Post Doctoral and the PhD Research Fellowships, respectively. Mahender Singh is supported by the Swarna Jayanti Fellowship grants DST/SJF/MSA-02/2018-19 and SB/SJF/2019-20/04.

\end{document}